\documentclass{article}
\usepackage{graphicx}
\usepackage{amsmath}
\usepackage{amscd}
\usepackage{amsfonts}
\usepackage{amssymb}

\newtheorem{theorem}{Theorem}[section]

\newtheorem{proposition}[theorem]{Proposition}

\newtheorem{definition}[theorem]{Definition}

\newtheorem{remarks}[theorem]{Remarks}

\newenvironment{proof}[1][Proof]{\textbf{#1.} }{\ \rule{0.5em}{0.5em}}

\newcommand{\be}{\begin{equation}}
\newcommand{\ee}{\end{equation}}
\newcommand{\bes}{\begin{equation*}}
\newcommand{\ees}{\end{equation*}}

\newcommand{\cC}{\mathcal{C}}

\begin{document}

\title{Conjugacy of P-configurations and nonlinear solutions to a certain conditional Cauchy equation}

\author{Orr Moshe Shalit \\Department of Mathematics, Technion \\
 Email: orrms@tx.technion.ac.il}

\date{}

\maketitle

\begin{abstract}
{We study the connection between conjugations of a special kind of dynamical systems, called \emph{P-configurations}, 
and solutions to homogeneous Cauchy type functional equations. We find that any two \emph{regular} P-configurations
are conjugate by a homeomorphism, but cannot be conjugate by a diffeomorphism. This leads us to the following conclusion (answering an open question  posed by Paneah):
\emph{there exist continuous nonlinear solutions to the functional equation:}
\bes
f(t) = f\left(\frac{t+1}{2}\right) + f\left(\frac{t-1}{2}\right) \,\, , \,\, t \in [-1,1] .
\ees
}
\end{abstract}

\section{Introduction}

A homogeneous Cauchy type functional equation is an equation of the form:
\be\label{eq:CTFE}
f(t) = f(\delta_1(t)) + f(\delta_2(t)) \,\, , \,\, t \in [-1,1] ,
\ee
where $f$ in an unknown function and $\delta_1,\delta_2$ are two increasing maps on $I$ that satisfy $\delta_1(t) + \delta_2(t) = t$ and certain additional conditions that will be detailed below. Two maps satisfying these conditions are said to form a \emph{P-configuration} in $[-1,1]$ (for a definition see Section 2). The solvability of this and of the accompanying non-homogeneous problem and its generalizations has been studied first by Paneah (see, e.g., \cite{P03}, \cite{P04}, \cite{POD}) and later by the author of this paper \cite{S05}. It is clear that the function $f(z) = cz$ is always a solution to (\ref{eq:CTFE}) and the interesting question is if and when there exist additional solutions. If the P-configuration happens to be regular (and in many other, much more general cases that we shall not discuss) then the only $C^1$ solutions
to (\ref{eq:CTFE}) are linear:  
\begin{theorem}\label{thm:P} {\bf (Paneah, \cite[Theorem 2]{POD})\footnote{We stress that \cite[Theorem 2]{POD} actually treats a much more difficult case where the derivatives $\delta_1'$ and $\delta_2'$ might also vanish sometimes.} .}
Assume that $\delta_1$ and $\delta_2$ form a P-configuration in $I$, and that $\delta_1'(t),\delta_2'(t) > 0$ for all $t \in [-1,1]$. If  $f \in C^1$ is a solution to (\ref{eq:CTFE}) then there is some $c \in \mathbb{R}$ such that
\bes
f(z) = cz \,\, , \,\, z \in [-1,1].
\ees
\end{theorem}

It has been an open question whether the above theorem is true without the assumption  $f \in C^1$. 
Even for the following simple looking equation:
\be\label{eq:P}
f(t) = f\left(\frac{t+1}{2}\right) + f\left(\frac{t-1}{2}\right) \,\, , \,\, t \in [-1,1] ,
\ee
it was not known whether there exist continuous solutions which are not linear functions.
One of the causes for interest in this question was the desire to know whether the $C^1$ condition in this theorem is
an artifact of the proof, or a true feature of the problem.

Another motive for studying the equations (\ref{eq:CTFE}) and (\ref{eq:P}) comes from the theory of conditional Cauchy equation, or more generally, the theory of \emph{redundancy} or \emph{overdeterminedness} (see \cite{S06} and the references therein). (This topic has been studied by many noted researchers in functional equation such as Dhombres, Forti, Ger, Jarczyk, Kuzma, Matkowski, Paneah, Sablik and others. We cannot give reference to all the works on this topic, but we must mention Kuzma \cite{K}, Dombres and Ger \cite{DG} and Chapters 6 and 16 of the book by Acz\'{e}l and  Dhombres \cite{AD}). It is an easy exercise to see that the only continuous functions $f:[-1,1] \rightarrow \mathbb{R}$
that satisfy the Cauchy Functional Equation
\be\label{eq:CK}
f(x + y) = f(x) + f(y) \,\, , \,\, (x, y) \in K ,
\ee
where $K = \{ (x,y) : |x| + |y| \leq 1 \}$, are linear functions. A less trivial result is that the only continuous solutions of the following equation:
\be\label{eq:CpK}
f(x + y) = f(x) + f(y) \,\, , \,\, (x, y) \in \partial K ,
\ee
are also linear functions, where $\partial K$ is the boundary of $K$ in $\mathbb{R}^2$ \cite[Section 3.1]{S05}. 
This shows that equation (\ref{eq:CK}) contains much more information than one really needs in order to determine $f$.
Equation (\ref{eq:CpK}) should be considered as the restriction of the functional equation (\ref{eq:CK}) to $\partial K$. 
In fact, if $\Gamma$ denotes the curve $\partial K \setminus \{(x,x+1) : x \in (-1,0)\}$, then it is true that the restriction of (\ref{eq:CK}) to $\Gamma$ does not admit any continuous solutions other than the linear ones.
It is natural to ask: \emph{What happens if we take a smaller subset of $\partial K$? Does the set of continuous solutions remain the same?} Note that (\ref{eq:P}) is nothing but (\ref{eq:CK}) restricted to one of the sides of the square $\partial K$.

In Section 3 below we solve the above mentioned open problem, by showing that there exist nonlinear continuous solutions
to every functional equation of the form (\ref{eq:CTFE}), and in particular to (\ref{eq:P}). This is achieved by an analysis of certain dynamical systems called \emph{P-configurations}. In Section 2 we prove that every two regular P-configurations are isomorphic. The isomorphism between two regular P-configurations is continuous but not continuously differentiable, as we show in Section 4. In Section 5 we show that although there is only ``one" regular P-configuration, the next simplest class of P-configurations has infinitely many isomorphism classes.

\section{Conjugacy of regular P-configurations}
We begin by giving the definition of a special kind of dynamical systems called \emph{P-configurations}. This notion was introduced by Paneah as a tool in solving several problems in functional equations, integral geometry and PDE (see \cite{P04}). The type of dynamics considered by Paneah was governed by \emph{guiding sets}, and general dynamical systems governed by guiding sets, called \emph{guided dynamical systems}, were studied in \cite{S05}.

Let $I = [a, b]$ be a fixed closed interval in $\mathbb{R}$, $c
\in (a,b)$, and let $\delta_1, \delta_2: I \rightarrow I$ be two
$C^1$ maps satisfying the following conditions:
\begin{equation}\label{eq:Pconf1}
\delta_1'(t) + \delta_2'(t) = 1 \quad , \quad t \in I \,\, ;
\end{equation}
\begin{equation}\label{eq:Pconf2}
\delta_i'(t) \geq 0 \quad , \quad t \in I, i = 1,2 \,\, ;
\end{equation}
\begin{equation}\label{eq:Pconf3}
\delta_2(a) = a, \quad \delta_2(b) = \delta_1(a) = c, \quad
\delta_1(b) = b \,\, .
\end{equation}
If all these assumptions hold, then the maps $\delta_1$ and
$\delta_2$ are said to form a P-configuration in $I$.
We introduce the guiding sets
\begin{displaymath}
\Lambda_1 = \{t \in I | \delta_1'(t) = 0\}
\end{displaymath}
and
\begin{displaymath}
\Lambda_2 = \{t \in I | \delta_2'(t) = 0\} .
\end{displaymath}
The sets $\Lambda_1, \Lambda_2$ are called \emph{guiding sets}. We shall also refer to the triple $(I,\delta, \Lambda)$ as the P-configuration, $\delta$ and $\Lambda$ being shorthand for $(\delta_1, \delta_2)$ and $(\Lambda_1, \Lambda_2)$.
If $\Lambda_1 = \Lambda_2 = \emptyset$ then we say that the P-configuration is \emph{regular}, and we denote it simply
by $(I,\delta)$.

\begin{remarks}
\begin{enumerate}
	\item It is clear that there is no essential loss if one assumes that $I = [-1,1]$ and $c = 0$, and we will assume it from now on.
	Equations (\ref{eq:Pconf1}) and (\ref{eq:Pconf3}) then take the form
	\begin{equation}\label{eq:Pconf1'}
  \delta_1(t) + \delta_2(t) = t \quad , \quad t \in I \,\, ,
  \end{equation} 
	\begin{equation}\label{eq:Pconf3'}
  \delta_2(-1) = -1, \quad \delta_2(1) = \delta_1(-1) = 0, \quad
  \delta_1(1) = 1 \,\, .
  \end{equation}
	\item In applications to analysis, the maps $\delta_1, \delta_2$ are often assumed to be \emph{twice} continuously 
	differentiable. We shall not need this assumption in this note.
\end{enumerate}
\end{remarks}

\begin{definition}
Two P-configurations $(I,\delta, \Lambda)$ and $(I,\sigma, \Omega)$ are said to be \emph{conjugate} (or \emph{isomorphic}) if there
is a homeomorphism $h: I \rightarrow I$ such that $h(\Lambda_i) = \Omega_i$ and
\bes
h(\delta_i(t)) = \sigma_i(h(t)) \,\, , \,\,i=1, 2\,\, , \,\, t \in I.
\ees
The map $h$ is called a \emph{conjugation} (or \emph{isomorphism}).
\end{definition}
Note that for regular P-configurations the notion of conjugacy is the same as the usual notion of conjugacy (or isomorphism) of dynamical systems.

\begin{proposition}\label{prop:conj}
Any two regular P-configurations are conjugate: if $(I,\delta_1, \delta_2)$ and 
$(I,\sigma_1, \sigma_2)$ are two regular P-configurations, then there exists a unique homeomorphism
$h:I \rightarrow I$ such that 
\be\label{eq:fixedpoints}
h(-1)=-1 \, , \, h(0)=0 \, , \, h(1)=1 
\ee
and
\be\label{eq:conj}
h(\delta_i(t)) = \sigma_i(h(t)) \,\, , \,\,i=1,2 \,\, , \,\, t \in I.
\ee
\end{proposition}
\begin{proof}
We may assume that $\sigma_1(t) = \frac{t+1}{2}$ and $\sigma_2(t) = \frac{t-1}{2}$, as conjugacy is an equivalence relation. Note that (\ref{eq:fixedpoints}) is a necessary condition for (\ref{eq:conj}) to hold (take $t=-1, i=2$, then $t=1, i=1$, and then $t=1, i=2$). 

The maps $\delta_1, \delta_2$ are invertible on their range, and so we may define two functions on $I$ as follows:
\bes
f(z) = \begin{cases}
\delta_2^{-1}(z) , & z \in [-1,0] \cr 
\delta_1^{-1}(z) , & z \in (0,1]  
\end{cases}
\ees
and 
\bes
\chi(z) = \begin{cases}
-1 , & z \in [-1,0] \cr 
1 , & z \in (0,1]  \,\, .
\end{cases}
\ees
$f$ should be thought of as an inverse of both $\delta_1, \delta_2$. We define 
\bes
\cC = \{g \in C([-1,1]):g(-1) = -1, g(0) = 0, g(1) = 1, \textrm{ g is nondecreasing}\}
\ees
endowed with the usual sup-norm metric. Define a mapping $T:\cC \rightarrow \cC$ by
\bes
(Tg)(z) = \frac{g(f(z)) + \chi(z)}{2}.
\ees
Note that $Tg$ is in $\cC$, and in particular it is a continuous function. One verifies directly that $h$ satisfies (\ref{eq:conj}) if and only if $Th = h$. Indeed, on each interval
$[-1,0]$ and $(0,1]$, this is seen by the applying change of variables $z = \delta_2(t)$ or $z = \delta_1(t)$.
But $T$ is clearly a strict contraction, hence by Banach's Fixed Point Theorem, there is a unique $h \in \cC$ that 
satisfies (\ref{eq:conj}).

It remains to show that $h$ is a homeomorphism. $I$ being a closed interval and $h$ being continuous, it suffices to show that $h$ is injective and surjective, the latter of which follows immediately from $h(-1) = -1, h(1) = 1$. Assume that $h$ is not injective. 
We will arrive at a contradiction and thereby finish the proof. Being nondecreasing, $h$ must be constant on some interval $J_0 \subseteq I$. Using (\ref{eq:conj}), one sees that $h$ must be constant on each of the intervals $\delta_1^{-1}(J_0)$ and $\delta_2^{-1}(J_0)$. In fact, unless $0 \in J_0$, only one of $\delta_1^{-1}(J_0)$ and $\delta_2^{-1}(J_0)$ is not empty. In any case, choose a nondegenerate interval from the two and denote it by $J_1$. Continuing this way, we iteratively construct a sequence of intervals $J_0, J_1, J_2, \ldots $ on which $h$ must be constant, by defining $J_{n+1}$ to be a nondegenerate interval, either $\delta_1^{-1}(J_n)$ or $\delta_2^{-1}(J_n)$. 

Assume first that $0$ is never contained in any $J_n$. One easily sees that the length of $J_{n+1}$ is at least $\rho^{-1}$ times the length of $J_n$, where $\rho$ is the maximum of $\sup_{t\in I}\{\delta_1'(t)\}$ and $\sup_{t\in I}\{\delta_2'(t)\}$. But $\rho < 1$, whence the length of $J_n$ must eventually be larger than $1$, so eventually $0 \in J_n$, a contradiction.

So there is some $n$ such that $0 \in J_n$. Without loss of generality, $n = 0$ and $J_1 := \delta_2^{-1}(J_0)$ is a nondegenerate interval containing $1$, and all the $J_n$'s for $n>1$ can be chosen to be $\delta_1^{-1}(J_{n-1})$ and will therefore also contain $1$. But then as long as $0 \notin J_n$, $J_{n+1}$ must be longer than $J_n$ by a factor of $\rho^{-1}$ as above. Thus eventually we must have $0,1 \in J_n$, and $h$ is constant on $J_n$. But $h(0)=0, h(1)=1$, so this is impossible.
\end{proof}

Examining the above proof one sees that no use has been made of condition (\ref{eq:Pconf1'}). 
We are lead to making the following definition.
\begin{definition}
Let $\delta_1, \delta_2 \in C^1(I)$ satisfy (\ref{eq:Pconf2}) and (\ref{eq:Pconf3}). 
Define $\Lambda_1 = \{t \in I | \delta_1'(t) = 0\}$ and $\Lambda_2 = \{t \in I | \delta_2'(t) = 0\}$. The (guided) dynamical system $(I,\delta,\Lambda)$ is said to be a \emph{quasi P-configuration}. If $\Lambda_1 = \Lambda_2 = \emptyset$ then $(I,\delta,\Lambda)$ is called a \emph{regular quasi P-configuration}.
\end{definition}
Combining the proof of Proposition \ref{prop:conj} with its conclusion, we conclude:
\begin{proposition}\label{prop:conj_gen}
Let $(I,\delta_1, \delta_2)$ and 
$(I,\sigma_1, \sigma_2)$ be two regular quasi P-configurations. Assume that for all $t\in I$, $\sigma_1(t),\sigma_2(t) < 1$. Then there exists a unique continuous map 
$h:I \rightarrow I$ such that 
\bes
h(-1)=-1 \, , \, h(0)=0 \, , \, h(1)=1 
\ees
and
\bes
h(\delta_i(t)) = \sigma_i(h(t)) \,\, , \,\,i=1,2 \,\, , \,\, t \in I.
\ees
If in addition $\delta_1(t),\delta_2(t) < 1$ for all $t\in I$, then $h$ is a homeomorphism, and $(I,\delta)$ and $(I,\sigma)$ are conjugate.
\end{proposition}
It remains an open question whether every two regular quasi P-configurations are conjugate or not.

\section{Existence of nonlinear solutions to homogeneous Cauchy type functional equations}

Using our observations from the previous section we can now resolve the open problem mentioned in the introduction.

\begin{theorem}\label{thm:exist}
Let $(I,\delta)$ be a regular P-configuration. Then equation (\ref{eq:CTFE}) has a nonlinear solution.
\end{theorem}
\begin{proof}
Let $(I,\sigma_1, \sigma_2)$ be another regular P-configuration with $\sigma_1 \neq \delta_1$. By Proposition \ref{prop:conj}, there exists a homeomorphism $h:I \rightarrow I$ such that 
\bes
h(\delta_i(t)) = \sigma_i(h(t)) \,\, , \,\,i=1,2 \,\, , \,\, t \in I.
\ees
Adding these two equations we obtain
\bes
h(\delta_1(t)) + h(\delta_2(t)) = \sigma_1(h(t)) + \sigma_2(h(t)) = h(t) \,\, , \,\, t \in I,
\ees
thus, $h$ solves (\ref{eq:CTFE}). Now $h(1)=1$, so if it were linear it would have to be the identity function on $I$. However, $h$ is not the identity function, because that would have implied $\delta_1(t) = h(\delta_1(t)) = \sigma_1(h(t)) = \sigma_1(t)$, contrary to the fact that $\sigma_1 \neq \delta_1$.
\end{proof}

The above proof shows that every conjugation of two P-configurations is a solution of (\ref{eq:CTFE}). 
There is also a partial converse to this. We record these facts in the following theorem.
\begin{theorem}
If $(I,\delta_1, \delta_2)$ and 
$(I,\sigma_1, \sigma_2)$ are two regular P-configurations, and $f:I \rightarrow I$ is a continuous function such that 
\be\label{eq:fconj}
f(\delta_i(t)) = \sigma_i(f(t)) \,\, , \,\,i=1,2 \,\, , \,\, t \in I,
\ee
then $f$ satisfies (\ref{eq:CTFE}). Conversely, if $(I,\delta_1, \delta_2)$ is a given regular P-configuration and 
$f$ is a strictly increasing continuous solution of (\ref{eq:CTFE}) that fixes the points $-1,0$ and $1$, then there exists a dynamical system
$(I,\sigma_1, \sigma_2)$, where $\sigma_1, \sigma_2$ are strictly increasing maps that satisfy (\ref{eq:Pconf1'}) 
and (\ref{eq:Pconf3'}), such that $f$ is a conjugation of $(I,\delta_1, \delta_2)$ and $(I,\sigma_1, \sigma_2)$.
\end{theorem}
We cannot say that $(I,\sigma_1, \sigma_2)$ is a P-configuration because the maps are not necessarily differentiable.

\begin{proof}
We have to prove the converse direction. Define
\bes
\sigma_i = f\circ\delta_i \circ f^{-1} .
\ees
Obviously, $f$ is a conjugation of $(I,\delta_1, \delta_2)$ and $(I,\sigma_1, \sigma_2)$. The $\sigma_i$'s are strictly increasing and satisfy (\ref{eq:Pconf3'}). To see (\ref{eq:Pconf1'}), use (\ref{eq:CTFE}) to compute
\bes
\sigma_1(t) + \sigma_2(t) = f(\delta_1(f^{-1}(t))) + f(\delta_2(f^{-1}(t))) = f(f^{-1}(t)) = t.
\ees
\end{proof}

We conclude this section with another application to functional equations.
\begin{theorem}
Let $(I,\delta)$ be a quasi regular P-configuration. Then the functional equation
\bes
f(t) = f(\delta_1(t)) + f(\delta_2(t)) \,\, , \,\, t \in I ,
\ees
has a nontrivial solution.
\end{theorem}
\begin{proof}
The proof is the same as the proof of Theorem \ref{thm:exist}, except that one uses Proposition \ref{prop:conj_gen} instead of \ref{prop:conj}.
\end{proof}

\section{Nonexistence of a continuously differentiable conjugacy between regular P-configurations}

By Proposition \ref{prop:conj} every two regular P-configurations are conjugate, i.e., there is a homeomorphism $h:I \rightarrow I$ that intertwines the maps. However, P-configurations are \emph{smooth} entities, the maps being assumed to be $C^1$, and sometimes smoother. Therefore, it is natural to ask whether the (unique) conjugating map $h$ is as smooth as the maps 
in the P-configurations, or is it at least continuously differentiable. It is remarkable that this \emph{never} happens, unless the two P-configurations are identical.

\begin{proposition}
Let $(I,\delta_1, \delta_2)$ and 
$(I,\sigma_1, \sigma_2)$ be two distinct (i.e., $\sigma_1 \neq \delta_1$) regular P-configurations.
Then the conjugating map $h$ given by Proposition \ref{prop:conj} is not continuously differentiable.
\end{proposition}
\begin{proof}
We have already observed in the proof of Theorem \ref{thm:exist} that $h$ is a nonlinear solution to the homogeneous Cauchy type functional equation (\ref{eq:CTFE}). However, by Theorem \ref{thm:P} discussed above, the only $C^1$ solutions of (\ref{eq:CTFE}) are linear. Thus, $h$ cannot be $C^1$.
\end{proof}

\section{On the diversity of non-regular P-configurations}

By Proposition \ref{prop:conj}, all regular P-configurations are ``the same". An evident necessary condition for two not-necessarily-regular P-configurations to be conjugate is that the guiding sets (the $\Lambda$'s) be homeomorphic, as well as their complements. As the set of all possible pairs $(\Lambda_1, \Lambda_2)$ is the set of all pairs of closed and disjoint sets in the interval, we see that the set of conjugacy classes of P-configurations is quite rich.

But the distinction between genuinely different kinds of P-configurations starts much before that: we will see soon that
there are infinitely many P-configurations $(I,\delta, \Lambda)$ with $\Lambda_2 = \emptyset$ and $\Lambda_1$ a singleton which are pairwise not isomorphic. Indeed, let $(I,\delta, \Lambda)$ and $(I,\sigma, \Omega)$ be two P-configurations with $\Lambda_1 = \{\lambda\}$, $\Omega_1 = \{\omega\}$, and the other guiding sets empty.
With Proposition \ref{prop:conj} in mind, we expect there to be at most one conjugation $h$ between these two systems, when considered without
their guiding sets. Following this heuristic, we expect there to be no freedom left for adjusting $h$ in such a manner that $h(\lambda) = \omega$. The following proposition makes this heuristic precise.

\begin{proposition}
There are infinitely many non-isomorphic P-configurations $(I,\delta, \Lambda)$ with $\Lambda_2 = \emptyset$ and $\Lambda_1$ consisting of a single point.
\end{proposition}
\begin{proof}
For every $n \in \mathbb{N}$, denote $J_n:= [1-1/2^n,1-1/2^{n+1}]$. Fix two integers $n \neq k$. Let $\delta_1(t) = \frac{t+1}{2}$ for all $t \in I$, except for $t \in J_n$.
In $J_n$ we let $\delta_1$ be a smooth perturbation of $\frac{t+1}{2}$ such that its derivative in this segment
is never equal to $1$ and is equal $0$ at precisely one point $\lambda$. $\delta_2(t)$ is defined to be $t - \delta_1(t)$. We then have $\Lambda_2 = \emptyset$ and $\Lambda_1 = \{\lambda\} \subset  \textrm{int}(J_n)$.

We define $\sigma_1$ in the same manner with the difference that $\sigma_1(t)$ is different from $\frac{t+1}{2}$ only in $J_k$, and there its derivative is never $1$ and is $0$ at precisely one point $\omega$. As before we have $\Omega_2 = \emptyset$ and $\Omega_1 = \{\omega\} \subset \textrm{int}(J_k)$.

Now assume that $h$ is an isomorphism of the dynamical systems $(I,\delta)$ and $(I,\sigma)$. We shall show that it cannot be an isomorphism of P-configurations (that is, it cannot be an isomorphism of guided dynamical systems). We already know that $h(0) = 0$.
Putting $t = 0$ in
\bes
h(\delta_1(t)) = \sigma_1(h(t)) ,
\ees
we have $h(1/2) = h(\delta_1(0)) = \sigma_1(h(0)) = \sigma_1(0) = 1/2$. Induction yields 
\bes
h(1-1/2^m) = 1-1/2^m \,\, , \,\, m \in \mathbb{N}.
\ees
It follows that $h$ maps $J_m$ to $J_m$, and thus cannot map $\lambda$ to $\omega$.
\end{proof}

\section*{Acknowledgments}
I am grateful to Boris Paneah for giving me the problem. Daniel Reem has read a preliminary version of the manuscript 
and suggested many improvements. I wish also to thank Dilian Yang for several intriguing discussions.
The partial funding granted by the Pollak fellowship is greatly acknowledged.

\end{document}